\documentclass[12pt]{amsart}
\usepackage[all]{xy}
\usepackage{amssymb}
\usepackage{amsthm}
\usepackage{hyperref}
\hypersetup{colorlinks=true,linkcolor=blue,citecolor=magenta}
\usepackage{amsmath}
\usepackage{amscd,enumitem}
\usepackage{verbatim}
\usepackage{eurosym}
\usepackage{float}
\usepackage{color}
\usepackage{dcolumn}
\usepackage[mathscr]{eucal}
\usepackage[all]{xy}
\usepackage{hyperref}
\usepackage{bbm}
\usepackage[textheight=8.5in, textwidth=6.5in]{geometry}
\newtheorem*{thm*}{Theorem}
\newtheorem*{conj*}{Conjecture}
\newtheorem{theorem}{Theorem}

\newtheorem*{remark}{Remark}

\newtheorem*{TwoRemarks}{Two Remarks}

\newtheorem*{example}{Example}
\newtheorem{corollary}[theorem]{Corollary}

\newcommand{\Z}{\mathbb{Z}}
\newcommand{\Q}{\mathbb{Q}}
\newcommand{\R}{\mathbb{R}}

\newcommand{\C}{\mathbb{C}}
\newcommand{\im}{{\text {\rm Im}}}

\numberwithin{equation}{section}

\begin{document}
\title{The Eichler integral of $E_2$ and $q$-brackets of $t$-hook functions}
\author{Ken Ono}
\address{Department of Mathematics, University of Virginia, Charlottesville, VA 22904}
\email{ko5wk@virginia.edu}

\thanks{The author thanks the support of the Thomas Jefferson Fund and the NSF
(DMS-1601306).}

\begin{abstract}
For functions $f: \mathcal{P}\mapsto \C$ on partitions, Bloch and Okounkov defined a 
power series $\langle f\rangle_q$ that is the ``weighted average'' of $f.$
As Fourier series in $q=e^{2\pi i z}$, such $q$-brackets generate the ring of quasimodular forms, and
the modular forms that are powers of Dedekind's eta-function. Using work of Berndt and Han, we build
modular objects from
$$
f_t(\lambda):= t\sum_{h\in \mathcal{H}_t(\lambda)}\frac{1}{h^2},
$$
weighted sums  over partition hook numbers that are multiples of $t$.
We find that $\langle f_t \rangle_q$ is the Eichler integral
of $(1-E_2(tz))/24,$ which we modify to construct a function $M_t(z)$ that enjoys weight 0 modularity properties.
As a consequence,
the non-modular Fourier series
 $$H_t^*(z):=\sum_{\lambda \in \mathcal{P}} f_t(\lambda)q^{|\lambda|-\frac{1}{24}}
 $$
inherits weight $-1/2$ modularity properties. These are sufficient to imply a Chowla-Selberg type result, generalizing the fact that weight $k$ algebraic modular forms evaluated at discriminant $D<0$ points $\tau$ are algebraic multiples of   $\Omega_D^k,$ the $k$th power of the canonical period.  If we let
$\Psi(\tau):=-\pi i \left(\frac{\tau^2-3\tau+1}{12\tau}\right)-\frac{\log(\tau)}{2},$ then for $t=1$ we prove that
$$
H_1^*(-1/\tau)-\frac{1}{\sqrt{-i\tau}}\cdot H_1^*(\tau)\in \overline{\Q}\cdot \frac{\Psi(\tau)}{\sqrt{\Omega_D}}.
$$
\end{abstract}
\dedicatory{For my friend Lance Littlejohn in celebration of his distinguished career}
\maketitle
\section{Introduction and statement of results}

A {\it partition} of a non-negative integer $n$ is any nonincreasing sequence of positive integers which sum to $n$. 
As usual, if $\lambda=\lambda_1+\lambda_2+\dots+\lambda_s$ is a partition of size $|\lambda|=n,$ then
we associate the {\it Ferrers-Young} diagram
$$
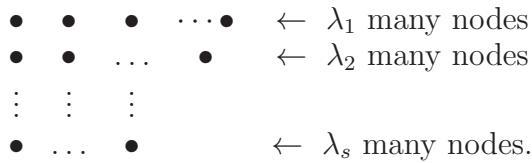
\begin{matrix}
\bullet & \bullet & \bullet & \dots \bullet &\ \leftarrow  \ {\text {\rm $\lambda_1$ many nodes}}\\
\bullet & \bullet &\dots &\bullet & \ \leftarrow \ {\text {\rm $\lambda_2$ many nodes}}\\
\vdots & \vdots & \vdots & \ & \ &  \\
\bullet & \dots & \bullet & \ & \   \leftarrow \ {\text {\rm $\lambda_s$ many nodes}}.
\end{matrix}
$$
Each node has a {\it hook number}. For a node in row $i$ and column $j,$ it 
is the positive integer $h(i,j):=(\lambda_i-i)+(\lambda'_j-j)+1,$ where $\lambda'_j$ denotes the number of nodes in column $j$.

Ferrers-Young diagrams and their hook numbers play significant roles in representation theory (for example, see \cite{KerberJames}). Indeed, partitions of size $n$ are used
to define {\it Young tableaux}, and their combinatorial properties encode the representation theory of the symmetric group $S_n$.
For example, the $t$-{\it core partitions} of size $n$ play an important role in number theory (for example, see \cite{G, GranvilleOno}) and the modular representation theory of $S_n$ and $A_n$ (for example, see
Chapter 2 of \cite{KerberJames}, and \cite{FongSrinivasan, GranvilleOno}). Recall that a partition is a $t$-core if none of its hook numbers 
 are multiples of $t$. If $p$ is prime, then the existence of a $p$-core of size $n$ is equivalent to the existence
of a defect 0 $p$-block for both $S_n$ and $A_n$. 

In this note we make use of partitions that are not $t$-core, namely those partitions that have some hook numbers that are multiples of  $t.$ We shall use these partitions to
 define partition functions $f_t : \mathcal{P}\mapsto \Q,$ which in turn we use to define
new modular objects.

To make this precise, we recall the framework of $q$-brackets.
For functions $f: \mathcal{P}\mapsto \C$ on the integer partitions, Bloch and Okounkov \cite{BlochOkounkov} defined the formal
power series 
\begin{equation}
\langle f\rangle_q:=\frac{\sum_{\lambda\in \mathcal{P}} f(\lambda)q^{|\lambda|}}{\sum_{\lambda\in \mathcal{P}}q^{|\lambda|}} \in \C[[q]],
\end{equation}
 which can conceptually be thought of as the ``weighted average'' of $f.$
 Schneider \cite{Schneider} developed a ``multiplicative theory of partitions'' based on these $q$-brackets,
 which includes partition theoretic analogs of many constructions in classical multiplicative number theory, such as M\"obius inversion,
 Dirichlet convolution, as well as incarnations of sieve methods.
 
Viewed as Fourier expansions, Bloch and Okounkov \cite{BlochOkounkov} showed that the ring of quasimodular forms is generated by the $q$-brackets of distinguished  functions $f$ associated to
shifted symmetric polynomials, work which was later expanded and refined by Zagier \cite{Zagier}.  Recently, Griffin, Jameson, and Trebat-Leder \cite{GJTL} developed a theory of $p$-adic modular forms in the context of these specific $q$-brackets for these shifted symmetric polynomials.

Nekrasov and Okounkov
\cite{NekrasovOkounkov} and Westbury \cite{Westbury}\footnote{Westbury discovered (see Prop. 6.1 and 6.2 of \cite{Westbury}) the Nekrasov-Okounkov formula concurrently.} later defined
further functions, say $D_s$, formulated in terms of partition hook numbers. They used these functions to give a partition theoretic description of every power (including complex) of Dedekind's eta-function $\eta(z):=q^{\frac{1}{24}}\prod_{n=1}^{\infty}(1-q^n).$ 
For $s\in \C$, define the function $D_s: \mathcal{P}\mapsto \C$ by
$$
D_s(\lambda):=\prod_{h\in \mathcal{H(\lambda)}} \left(1-\frac{s}{h^2}\right),
$$
where $\mathcal{H}(\lambda)$ denotes the multiset of hook numbers of the partition $\lambda$.
A slight reformulation of their formula (see (6.12) of \cite{NekrasovOkounkov}), using Euler's partition generating function 
\begin{equation}\label{Euler}
\sum_{n=0}^{\infty}p(n)q^n=\sum_{\lambda \in \mathcal{P}}q^{|\lambda|}=\prod_{n=1}^{\infty}\frac{1}{(1-q^n)},
\end{equation}
 asserts that
$q^{\frac{s}{24}}\cdot \langle D_s \rangle_q= \eta(z)^s.$
If $s\in \Z,$ then we have
a weight $s/2$ modular form (see Chapter 1 of \cite{CBMS}), thereby encoding a particularly important family of modular forms.

In view of these constructions of quasimodular forms and powers of Dedekind's eta-function, it is natural to ask
whether further modular objects arise from $q$-brackets. To this end, we recall one of the simplest $q$-brackets, the formal
power series associated to  $f(\lambda):=|\lambda|,$ the ``size'' function. A simple calculation using (\ref{Euler})
shows that
\begin{equation}\label{simple}
\langle |\cdot|\rangle_q=\frac{\sum_{\lambda\in \mathcal{P}} |\lambda|q^{|\lambda|}}{\sum_{\lambda\in \mathcal{P}}q^{|\lambda|}}=
\sum_{n=1}^{\infty}\sigma_1(n)q^n=\frac{1-E_2(z)}{24},
\end{equation}
where $\sigma_{v}(n):=\sum_{1\leq d\mid n}d^v$, and $E_2(z)=1-24\sum_{n=1}^{\infty}\sigma_1(n)q^n$ is the weight 2 quasimodular Eisenstein series. Although $E_2(z)$ is not a modular form, it is well known that
$$
E_2^*(z):=-\frac{3}{\pi\cdot \im(z)}+E_2(z)
$$
is a non-holomorphic weight 2 modular form (for example, see Chapter 6 of \cite{BFOR}). Therefore, its {\it modified
Eichler integral}
\begin{equation}\label{E}
\mathcal{E}(z):=\sum_{n=1}^{\infty}\frac{\sigma_1(n)}{n}\cdot q^n=\sum_{n=1}^{\infty}\sigma_{-1}(n)q^n
\end{equation}
enjoys certain weight 0 modularity properties. These were determined by Berndt in the 1970s \cite{Berndt}.
Therefore, in view of (\ref{simple}), it is natural to ask whether $\mathcal{E}(z)$ arises as a $q$-bracket. 

We show that this is indeed the case, and the construction makes use of $t$-hooks, the hook numbers which are
multiples of $t$. To this end, for each  $t\in \Z^{+}$ we define $f_t: \mathcal{P}\mapsto \Q$ by
\begin{equation}
f_t(\lambda):=t \sum_{h\in \mathcal{H}_t(\lambda)}\frac{1}{h^2},
\end{equation}
where $\mathcal{H}_t(\lambda)$ is the multiset of hook numbers which are multiples of $t$. 

\begin{example} We consider the partition $\lambda=4+3+1,$
which has Ferrers-Young diagram
$$
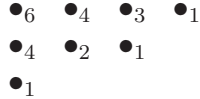
\begin{matrix} \bullet_6 & \bullet_4 & \bullet_3 & \bullet_1 \\
           \bullet_4 & \bullet_2 & \bullet_1\\
           \bullet_1 
           \end{matrix}
$$         
(the subscripts denote the the hook numbers). We find that
$\mathcal{H}(\lambda)=\{1,1,1,2,3,4,4,6\},$ $\mathcal{H}_2(\lambda)=\{2,4,4,6\},$ 
and $\mathcal{H}_3(\lambda)=\{3,6\}.$ Therefore, we find that
\begin{displaymath}
\begin{split}
f_1(\lambda)&= 1+1+1+\frac{1}{4}+\frac{1}{9}+\frac{1}{16}+\frac{1}{16}+\frac{1}{36}=\frac{253}{72},\\
f_2(\lambda)&=2\left(\frac{1}{4}+\frac{1}{16}+\frac{1}{16}+\frac{1}{36}\right)=\frac{29}{36},\\
f_3(\lambda)&=3\left(\frac{1}{9}+\frac{1}{36}\right)=\frac{5}{12}.
\end{split}
\end{displaymath}
\end{example} 

For convenience, we define
\begin{equation}\label{Ht}
H_t(z):=\sum_{\lambda\in \mathcal{P}} f_t(\lambda)q^{|\lambda|}.
\end{equation}
Work of Han \cite{Han} allows us to describe the $q$-brackets of the $f_t$ in terms 
of $\mathcal{E}(z).$

\begin{theorem}\label{Theorem1}
If $t$ is a positive integer, then we have
$$
\langle f_t\rangle_q=\prod_{n=1}^{\infty} (1-q^n)\cdot H_t(z)=\mathcal{E}(tz).
$$
\end{theorem}

\begin{remark}
In a follow-up to the present note, the author, Bringmann, and Wagner \cite{OW} will describe a general framework for
Eichler integrals of arbitrary weight Eisenstein series as $q$-brackets of ``weighted'' $t$-hook functions on partitions.
These results will include results where these  $q$-brackets are completed to obtain harmonic Maass forms, sesquiharmonic Maass forms (in the case of $E_2$), and holomorphic quantum modular forms.

\end{remark}

We now turn to the problem of determining the modularity properties of these $q$-brackets.
Using the crucial functions $P_t(z)$ and $L_t(z)$ defined by
\begin{equation}
P_t(z):=-t\left( t+\frac{\pi i}{12}\right)z +\frac{1}{z} \ \ \ \ \ {\text {\rm and}}\ \ \ \ \
L_t(z):=-\frac{1}{4}\cdot \log(tz),
\end{equation}
we define
\begin{equation}
M_t(z):=P_t(z)+L_t(z)+\langle f_t\rangle_q.
\end{equation}
These functions enjoy weight 0 modularity properties with respect to translation
$z\rightarrow z+1$ and inversion $z\rightarrow -1/t^2z.$

\begin{theorem}\label{Theorem2}
If $t$ is a positive integer, then the following are true for all $z\in \mathbb{H}.$

\begin{enumerate}
\item We have that
$$
M_t(z+1)-M_t(z)=-t\left(t+\frac{\pi i}{12}\right)-\frac{1}{z(z+1)}+\frac{1}{4}\log\left(\frac{z}{z+1}\right).
$$
\item We have that
$$
M_t(z)=M_t\left(-\frac{1}{t^2z}\right).
$$
\end{enumerate}
\end{theorem}

\begin{TwoRemarks} \ \newline
\noindent
1) We use the branch of $\log z$ with $-\pi \leq \arg(z) <\pi$
(resp. $\sqrt{z}$ that is positive on $\R^{+}$).

\noindent
2) Obviously, $\langle f_t\rangle_q=\mathcal{E}(tz)$ is invariant under $z\rightarrow z+1.$ The definition of
$M_t(z)$ is motivated by the desire to obtain the more difficult invariance under the inversion $z\rightarrow -1/t^2 z$.
\end{TwoRemarks}

Theorem~\ref{Theorem2} implies modularity properties for the individual series $H_t(z)$ defined in (\ref{Ht}).
We modify these generating functions to define
\begin{equation}
H_t^*(z):=q^{-\frac{1}{24}}\cdot H_t(z)=\sum_{\lambda\in \mathcal{P}}f_t(\lambda)q^{|\lambda|-\frac{1}{24}},
\end{equation}
so that $\langle f_t\rangle_q=\eta(z)\cdot H_t^*(z).$ This allows us to make use of the modularity
of Dedekind's eta-function. To ease notation, we define the auxiliary function
\begin{equation}\label{Psi}
\Psi(z):=-\pi i \cdot \left(\frac{z^2-3z+1}{12z}\right)-\frac{1}{2}\log(z).
\end{equation}
For $H_1^*(z),$ these transformation laws  are described in terms of $\Psi(z)$ and Dedekind's $\eta(-1/z).$

\begin{corollary}\label{H1star}
If $z\in \mathbb{H}$, then the following are true.

\begin{enumerate}
\item We have that
$$
H_1^*(z+1)=e^{-\frac{\pi i}{12}}\cdot H_1^*(z).
$$
\item We have that
$$
H_1^*(-1/z)-\frac{1}{\sqrt{-iz}}\cdot H_1^*(z)=\frac{\Psi(z)}{\eta(-1/z)}.
$$
\end{enumerate}
\end{corollary}

These results imply a Chowla-Selberg type result, generalizing the classical fact \cite{ChowlaSelberg,vW} that weight $k$ algebraic modular forms
evaluated at discriminant $D<0$ points $\tau$ are algebraic multiples of $\Omega_D^k,$ the $k$th power of a canonical
period $\Omega_D.$
To make this precise, we let $\overline{\mathbb{Q}}$ denote the algebraic closure of the field of rational numbers.  Suppose that $D<0$ is the fundamental discriminant of the imaginary quadratic field $\mathbb{Q}(\sqrt{D})$.  Let $h(D)$ denote the class number of $\mathbb{Q}(\sqrt{D})$, and define $h'(D):=1/3$ (resp. $1/2$) when $D=-3$ (resp. $-4$), and $h'(D):=h(D)$ when $D<-4$.  We
then have the canonical period
\begin{equation}
\Omega_D:=\frac{1}{\sqrt{2\pi|D|}}\left(\prod_{j=1}^{|D|-1}\Gamma\left(\frac{j}{|D|}\right)^{\chi_D(j)}\right)^{\frac{1}{2h'(D)}},
\end{equation}
where $\chi_D(\bullet):=\left(\frac{D}{\bullet}\right)$. 

The Chowla-Selberg phenomenon concerns modular forms $f(z)$ with algebraic Fourier coefficients.
If $f(z)$ has weight $k\in \Z$ and $\tau\in\mathbb{H}\cap\mathbb{Q}(\sqrt{D}),$ then
\begin{equation}\label{CSTheorem}
 f(\tau)\in\overline{\mathbb{Q}}\cdot\Omega_D^k.
 \end{equation}
 We obtain the following generalization of this phenomenon for $H_1^*(-1/z)-\frac{1}{\sqrt{-iz}}\cdot H_1^*(z),$
 which is somewhat analogous to similar results obtained by Dawsey and the author \cite{DO} in a different partition theoretic
 setting.

\begin{corollary}\label{ChowlaSelbergResult}
If $\tau\in \Q(\sqrt{D})\cap \mathbb{H},$ where $D<0$ is a fundamental discriminant, then
$$
H_1^*(-1/\tau)-\frac{1}{\sqrt{-i\tau}}\cdot H_1^*(\tau)
\in \overline{\Q}\cdot \frac{\Psi(\tau)}{\sqrt{\Omega_D}}.
$$
\end{corollary}

\begin{remark} Corollary~\ref{ChowlaSelbergResult} is a special case of the general fact that 
$$
\alpha_t(\tau) H_t^*(-1/t^2 \tau)-\frac{\beta_t(\tau)}{\sqrt{-i\tau}}\cdot H_t^*(\tau)
\in \overline{\Q}\cdot \frac{\Psi(t\tau)}{\sqrt{\Omega_D}},
$$
where $\alpha_t(\tau):=\eta(-1/t^2\tau)/ \sqrt{\Omega_D}$ and $\beta_t(\tau):=\eta(-1/\tau)/\sqrt{\Omega_D}$ are both algebraic.
\end{remark}

To obtain these results, we make use of recent work of Han \cite{Han} on extensions of the Nekrasov-Okounkov formula, and work of Berndt \cite{Berndt}
on modular transformation properties for generalized Lambert series. These results are recalled in Section~\ref{NutsAndBolts},
and in Section~\ref{Proofs} we prove the theorems and corollaries described above.

\section*{Acknowledgements}\noindent  The author thanks Madeline Locus Dawsey, Wei-Lun Tsai, Ian Wagner, and Ole Warnaar for their comments on an earlier draft of this note.

\section{Nuts and Bolts}\label{NutsAndBolts}

Here we recall important work by Berndt and Han which are the critical ingredients in the proofs of the results obtained in this note.

\subsection{A formula of Han}

Han recently expanded and refined the Nekrasov-Okounkov product formula. 
Theorem~1.3 of \cite{Han} offers the striking refinement we require.

\begin{theorem}{\text {\rm (Han, 2008)}}\label{HanTheorem}
We have that
$$
\sum_{\lambda\in \mathcal{P}}q^{|\lambda|}\prod_{h\in \mathcal{H}_t(\lambda)} \left(y-\frac{tyw}{h^2}\right)=\prod_{n=1}^{\infty}\frac{(1-q^{tn})^t}{(1-(yq^t)^n)^{t-w}(1-q^n)}.
$$
\end{theorem}

\begin{remark}
Han's formula can be reformulated in terms of the $t$-hook function
$$
F_{t,y,w}(\lambda):=\prod_{h\in \mathcal{H}_t(\lambda)} \left(y-\frac{tyw}{h^2}\right).
$$
Theorem~\ref{HanTheorem} is equivalent to the identity
$$
\langle F_{t,y,w}\rangle_q=\prod_{n=1}^{\infty}\frac{(1-q^{tn})^t}{(1-(yq^t)^n)^{t-w}}.
$$
\end{remark}

\subsection{A formula of Berndt}

Using the modularity of the weight 2 nonholomorphic Eisenstein series $E_2^*(z),$ one knows that
$\mathcal{E}(z)$ possesses weight 0 type modular transformation laws which can be computed using the method
of Eichler integrals (for example, see Section 1.4 of \cite{BFOR}). It turns out that Berndt
\cite{Berndt} has previously determined the modular transformation properties we require.
Here we offer a slight reformulation of the $m=0$ case of
Theorem~2.2 of \cite{Berndt}, the key transformation property of $\mathcal{E}(z)$.

\begin{theorem}\label{BerndtTheorem}{\text {\rm (Berndt, 1977)}}
If $z\in \mathbb{H}$, then
$$
\mathcal{E}(z)-\mathcal{E}(-1/z)=-\Psi(z).
$$
\end{theorem}

\section{Proofs of Results}\label{Proofs}
Here we employ the work of Berndt and Han to prove our results. 

\begin{proof}[Proof of Theorem~\ref{Theorem1}]
Letting $y=1$ in Theorem~\ref{HanTheorem}, we find that
\begin{equation}\label{formula}
\sum_{\lambda\in \mathcal{P}}q^{|\lambda|}\prod_{h\in \mathcal{H}_t(\lambda)} \left(1-\frac{tw}{h^2}\right)=\prod_{n=1}^{\infty}\frac{1}{(1-q^{tn})^{-w}(1-q^n)}.
\end{equation}
We recall the classical identity for Euler's partition generating function
$$
\sum_{n=0}^{\infty}p(n)q^n=\prod_{n=1}^{\infty}\frac{1}{1-q^n}=\exp\left(
\sum_{n=1}^{\infty}\frac{q^n}{n(1-q^n)}\right).
$$
Applying this identity to the factor in (\ref{formula}) with exponent $-w$, we find that
$$
\sum_{\lambda \in \mathcal{P}}q^{|\lambda|}\prod_{h\in \mathcal{H}_t(\lambda)} \left(1-\frac{tw}{h^2}\right)=
\prod_{n=1}^{\infty}\frac{1}{1-q^n}\cdot \exp\left(-w\sum_{n=1}^{\infty}\frac{q^{tn}}{n(1-q^{tn})}\right).
$$
By comparing the coefficients in $w,$ we find from (\ref{Ht}) that
$$
H_t(z)=\prod_{n=1}^{\infty}\frac{1}{1-q^n}\cdot \sum_{n=1}^{\infty}\frac{q^{tn}}{n(1-q^{tn})}.
$$
A straightforward calculation shows that this is equivalent to the  claim that
$\langle f_t\rangle_q=\mathcal{E}(tz).$
\end{proof}

\begin{proof}[Proof of Theorem~\ref{Theorem2}]
By Theorem~\ref{Theorem1}, we have that
$$
M_t(z)=P_t(z)+L_t(z)+\mathcal{E}(tz).
$$
By letting $z\rightarrow tz$ in Theorem~\ref{BerndtTheorem}, we have that
$$
\mathcal{E}(tz)-\mathcal{E}(-1/tz)=-\Psi(tz).
$$
Therefore, by direct calculation we find that
\begin{displaymath}
\begin{split}
M_t(z)-M_t(-1/t^2 z)&=(P_t(z)-P_t(-1/t^2z))+(L_t(z)-L_t(-1/t^2z))+(\mathcal{E}(tz)-\mathcal{E}(-1/tz))\\
&=-\pi i\left(\frac{t^2z^2+1}{12tz}\right)+\frac{\pi i}{4}-\frac{1}{2}\log(tz) -\Psi(tz)
=0.
\end{split}
\end{displaymath}
This confirms the second claim.

The series $\mathcal{E}(tz)$ is invariant in $z\rightarrow z+1,$ as it 
 is a power series in $q=e^{2\pi i z}$ with integer exponents. Therefore, we find that
 $$
 M_t(z+1)-M_t(z)=\left( P_t(z+1)+L_t(z+1)\right)- \left(P_t(z)+L_t(z)\right).
 $$
 The first claim follows by direct calculation.
\end{proof}

\begin{proof}[Proof of Corollary~\ref{H1star}]
By definition, we have that
$$
H_1^*(z)=\sum_{\lambda \in \mathcal{P}}f_t(\lambda)q^{|\lambda|-\frac{1}{24}}.
$$
Therefore, the first claim is a triviality.

To establish the second claim, we note that Theorem~\ref{Theorem2}  implies that
$$
M_1(z)=M_1(-1/z).
$$
Using the fact that $\eta(-1/z)=\sqrt{-iz}\cdot \eta(z)$,
we obtain
$$
P_1(z)+L_1(z)+\frac{\eta(-1/z)}{\sqrt{-iz}}\cdot H_1^*(z)=P_1(-1/z)+L_1(-1/z)+\eta(-1/z)H_1^*(-1/z).
$$
Since the $\eta$-function is nonvanishing on $\mathbb{H}$, this is equivalent to the desired conclusion
\begin{displaymath}
\begin{split}
H_1^*(-1/z)-\frac{1}{\sqrt{-iz}}\cdot H_1^*(z)&=\frac{1}{\eta(-1/z)}\cdot \left(P_1(z)+L_1(z)-P_1(-1/z)-L_1(-1/z)\right)\\
&=\frac{\Psi(z)}{\eta(-1/z)}.
\end{split}
\end{displaymath}
\end{proof}

\begin{proof}[Proof of Corollary~\ref{ChowlaSelbergResult}]
By the classical Chowla-Selberg theorem described by (\ref{CSTheorem}), since the Dedekind eta-function has weight 1/2, we have that
$$
\eta(-1/\tau)\in \overline{\Q}\cdot \sqrt{\Omega_D}.
$$
The claimed conclusion is now an immediate consequence of Corollary~\ref{H1star} (2).
\end{proof}

\section{Some examples}
Here we offer examples of the results of this paper.

\begin{example} We illustrate the $t=1$ and $2$ cases of Theorem~\ref{Theorem1}.
By direct calculation, we have
\begin{displaymath}
\begin{split}
H_1(z)&=q+\frac{5}{2}q^2+\frac{29}{6}q^3+\frac{109}{12}q^4+\dots,\\
H_2(z)&=q^2+q^3+\frac{7}{2}q^4+\frac{9}{2}q^5+\dots.
\end{split}
\end{displaymath}
Therefore, one finds that
\begin{displaymath}
\begin{split}
\langle f_1\rangle_q = \prod_{n=1}^{\infty}(1-q^n)\cdot H_1(z)=q+\frac{3}{2}q^2+\frac{4}{3}q^3+\frac{7}{4}q^4+\dots =\mathcal{E}(z)
=\sum_{n=1}^{\infty}\sigma_{-1}(n)q^n,\\
\langle f_2\rangle_q = \prod_{n=1}^{\infty}(1-q^n)\cdot H_2(z)=q^2+\frac{3}{2}q^4+\frac{4}{3}q^6+\frac{7}{4}q^8+\dots =\mathcal{E}(2z)=\sum_{n=1}^{\infty}\sigma_{-1}(n)q^{2n}.
\end{split}
\end{displaymath}

\end{example}

\begin{example}
We illustrate Theorem~\ref{Theorem2} (2) with $t=2$ and $z=i.$
Ramanujan proved (see p. 326 of \cite{Berndt1998}) that
\begin{displaymath}
\begin{split}
\eta(i)&=\frac{\sqrt{2}\pi^{\frac{1}{4}}}{2\cdot \Gamma(3/4)} \approx 0.7682\\
\eta(i/4)&=\frac{2^{\frac{1}{4}}\cdot \sqrt{\pi}}{2\cdot \Gamma(3/4)^2} \approx 0.7018.
\end{split}
\end{displaymath}
Moreover, we find that
$$
H_2^*(i)\approx 4.5395\cdot 10^{-6}\ \ \ {\text {\rm and}}\ \ \ H_2^*(i/4)\approx 0.06572.
$$
Therefore, we find that
\begin{displaymath}
\begin{split}
M_2(i)&=P_2(i)+L_2(i)+\eta(i)H_2^*(i)\approx 0.3503-5.3926 i,\\
M_2(i/4)&=P_2(i/4)+L_2(i/4)+\eta(i/4)H_2^*(i/4)\approx 0.3503-5.3926i.
\end{split}
\end{displaymath}
This illustrates the fact that $M_2(i)=M_2(i/4).$
\end{example}

\begin{example}
We now illustrate Corollary~\ref{H1star} (2) and Corollary~\ref{ChowlaSelbergResult} using $z=\tau=2i.$
By direct calculation, we find that $H_1^*(i/2)\approx 0.05506$ and $H_2^*(2i)\approx 5.8870\cdot 10^{-6}.$ Therefore, we have
\begin{equation}\label{formula2}
H_1^*(i/2)-\frac{\sqrt{2}}{2}\cdot H_1^*(2i)\approx 0.05506.
\end{equation}
Ramanujan proved (see p. 326 of \cite{Berndt1998}) that
$$
\eta(i/2)=2^{\frac{1}{8}}\cdot \sqrt{\Omega_{-4}}=\frac{\pi^{\frac{1}{4}}}{2^{\frac{3}{8}}\cdot \Gamma(3/4)}\approx 0.8377\dots.
$$
By direct calculation, we find that $\Psi(2i)=\frac{\pi}{8}-\frac{\log(2)}{2}\approx 0.04612\dots,$ and so
$$
\frac{\Psi(2i)}{\eta(i/2)}\approx 0.05506\dots.
$$
Combined with (\ref{formula2}), we find that
$$
H_1^*(i/2)-\frac{\sqrt{2}}{2}\cdot H_1^*(2i)=\frac{\Psi(2i)}{\eta(i/2)}=\frac{1}{2^{\frac{1}{8}}}\cdot \frac{\Psi(2i)}{\sqrt{\Omega_{-4}}}.
$$
This illustrates Corollary~\ref{ChowlaSelbergResult}, where the algebraic factor is $1/2^{\frac{1}{8}}.$
\end{example}

\end{document}